\documentclass[12pt,a4paper]{article}

\usepackage{latexsym,amssymb,amsthm}
\usepackage{url,graphics}
\usepackage[dvips]{graphicx,epsfig}

\title{\bf Upper bound on  the number of edges of an almost planar bipartite graph}
\author{
         D.\,V.\,Karpov\thanks{This research was supported by Russian Foundation for Basic Research (RFBR) grant 11-01-00760-a
and Grant NSh-3229.2012.1.}
     \\[6pt]
          {\small E-mail: \texttt{dvk0@yandex.ru}       }
    }
\date{}

\begin{document}
\maketitle
\righthyphenmin=2
\renewcommand*{\proofname}{\bf Proof}
\newtheorem{thm}{Theorem}
\newtheorem{lem}{Lemma}
\newtheorem{cor}{Corollary}
\theoremstyle{definition}
\newtheorem{defin}{Definition}
\theoremstyle{remark}
\newtheorem{rem}{\bf Remark}

\def\N{{\rm N}}
\def\cr{{\rm cr}}
\def\q#1.{{\bf #1.}}
\def\I{{\rm Int}}
\def\R{{\rm Bound}}
\def\mmin{\mathop{\rm min}}

\centerline{\sc Abstract}
Let $G$ be a bipartite graph without loops and multiple edges on~$v\ge 4$ vertices, which can be drawn on the plane such that any edge intersects at most one other edge.  We prove that such a graph has at most~$3v-8$ edges for even~$v\ne 6$ and at most~$3v-9$ edges for odd~$v$  and~$v=6$. For all~$v\ge 4$ examples showing that these bounds are tight are constructed.

In the end of the paper we discuss a question about  drawing of complete bipartite graphs on the plane such that any edge intersects at most one other edge.

\section{\bf Introduction}

We consider graphs without loops and multiple edges and use standard notations.
We denote the vertex set of a graph~$G$ by~$V(G)$ and its edge set by~$E(G)$. We use the notations~$v(G)$ and~$e(G)$ for the number of vertices and edges of a graph~$G$, respectively.

We denote by~$d_G(x)$ the degree of a vertex~$x$ in the graph~$G$. The minimal and maximal  vertex degree of a graph~$G$ are  denoted by~$\delta(G)$ and~$\Delta(G)$, respectively.

In many papers (for example,~\cite{PT}) bounds on the crossing number of plane drawings of graphs are discussed. We will not review these results here. Our paper is devoted to a similar but less popular question about drawing graphs on the plane such that any edge  intersects bounded number of other edges.

\begin{defin}
$1)$ Let $k$ be a nonnegative integer. We say that a graph is $k$-{\it planar}, if it can be drawn on the plane such that any edge intersects at most~$k$ other edges.

For~$k=1$  we call such graph  {\it almost planar}.

$2)$ Let~$\cr(G)$ denote the minimal number of pairs of crossing edges in a plane drawing of~$G$.
\end{defin}

Clearly, a $0$-planar graph is a planar graph. As usual, while speaking about plane  drawing of a graph, we assume that vertices of this graph are points and its edges are  ``good'' nonintersecting curves. Any such curve do not contain vertices different from the ends of correspondent edge. No three curves have a common inner point, i.e. each point of intersection belongs to exactly two edges.

A well known classic fact tells us that a planar graph on~$v$ vertices  contains at most~$3v-6$ edges. This bound is tight for any~$v\ge 3$.
It is proved in~\cite{PT}, that $e(G)\le (k+3)(v(G)-2)$ for a $k$-planar graph, where ${1\le k\le 4}$. It is shown  in~\cite{PT} that for $k=1$ the bound is attained for all~${v\ge 12}$. For~$k=2$ the bound is attained for sufficiently large~$v \equiv 2 \pmod{3}$. It is also shown in~\cite{PT}, that for any $k$-planar graph~$G$ (where~$k>0$) the inequality~$e(G)\le 4.108\sqrt{k}\cdot v(G)$ holds.

There are other works discussing similar questions, but their authors consider more complicated classes of graphs. For example, {\it $k$-quasi-planar graphs} (graphs, which can be drawn on the plane such that the intersection graph of this drawing do not contain a clique of size~$k$, see~\cite{AA}-\cite{A}). Graphs of this class   have much in common with  $k$-planar graphs, but they have more complicated structure.  Let us also mention several papers, devoted to questions of drawing graphs with constraint on some plane configurations:~\cite{TT},~\cite{PGT} and others.

Let us return to rather simple class of graphs~--- almost planar graphs. 
For these graphs we have the inequality~$e(G)\le 4v(G)-8$.  A well known classic fact tells us  that a bipartite planar graph on~$v$ vertices can have not more than~$2v-4$ edges and this bound is tight for any~$v\ge 3$.
In our paper we study bipartite almost planar graphs. We shall find a precise upper bound for the number of edges of a bipartite almost planar graph for any number of vertices~$v\ge 4$.

\begin{thm}
\label{dap} 
Let~$\beta(v)$ denote the maximal number of edges in a bipartite almost planar graph on~$v$ vertices for any~$v\ge 4$.  Then $\beta(v)=3v-8$ for even~$v\ne 6$ and  $\beta(v)=3v-9$ for odd~$v$ and for~$v=6$.
\end{thm}

\section{\bf Drawing of a graph}

An almost planar graph can be drawn on the plane in different ways. To avoid confusion we denote a drawing of a graph~$G$ in another font:~$\cal G$.  We demand that a drawing fulfill some requirements. Let us begin with four conditions.\\

$1^\circ$. {\it Any edge intersect at most one other edge.}\\

$2^\circ$. {\it No edge intersects itself.}\\

$3^\circ$. {\it Two intersecting edges do not have a common end.}\\

$4^\circ$. {\it Any two intersecting edges intersect each other in exactly one point.}\\

\begin{lem}
For any almost planar graph~$G$ there exists a plane drawing~$\cal G$, that satisfies conditions~$1^\circ-4^\circ$.
\end{lem}

\begin{proof}
Consider all drawings of the graph~$G$, satisfying the condition~$1^\circ$ (clearly, such  drawings exist). We choose among them a drawing~$\cal G$ with minimal  number of points of intersection of edges. 

Let us prove, that~$\cal G$ satisfies the condition~$2^\circ$. Assume the contrary and consider an edge~$e=ab$, which intersects itself.
Let this edge passes a point~$Z$ more than once. Then one can delete the part of drawing of~$e$ between the first and the last visit to the point~$Z$ (in direction from~$a$ to~$b$). As a result we obtain a drawing with less number of intersections. That contradicts our assumption.

Let us prove, that~$\cal G$ satisfies the conditions~$3^\circ$ and~$4^\circ$.
Assume the contrary, let~$\cal G$ does not satisfy any of these conditions. Then there exist two edges~$u_1u_2$ and~$v_1v_2$, having at least 
two  points of intersection.  Clearly, these edges cannot intersect other edges. Let us go along the edge~$v_1v_2$ and fix  the first and the last 
points of intersection of the edges~$u_1u_2$ and~$v_1v_2$ (let them be~$X$ and~$Y$, respectively, see figure~\ref{ap0}a). If the edges~$u_1u_2$ and~$v_1v_2$ have a common end, then they have only one common end. In this case let~$Y$ be the common end of these edges.

Let us mark  on the part~$v_1X$ of the edge~$v_1v_2$ a point~$X'$ very closely to~$X$ and replace the $X'Y$-part  of the edge~$v_1v_2$ by a $X'Y$-path~$S$ along the edge~$u_1u_2$ (see dotted line on figure~\ref{ap0}a). This path do not intersect any edge  (besides the replaced part of the edge~$v_1v_2$, but that does not matter). Thus we get rid of intersection in the point~$X$ and reduce the number of points of intersection of edges by at least~1. That  contradicts to the choice of the drawing~$\cal G$. 
\end{proof}

{\sf In what follows all drawings of almost planar graphs satisfy the conditions~${1^\circ-4^\circ}$.}

\begin{defin}
A bipartite almost planar graph~$G$ is called {\it maximal}, if after adding any edge it becomes not bipartite or not almost planar.
\end{defin}

\begin{figure}[!ht]
	\centering
		\includegraphics[width=0.9\columnwidth, keepaspectratio]{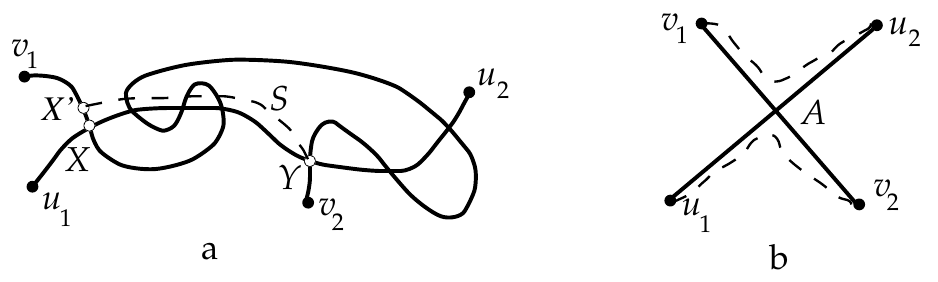}
     \caption{Intersecting edges.}
	\label{ap0}
\end{figure}

\smallskip
Consider a drawing~$\cal G$  of a bipartite almost planar graph~$G$ and its proper vertex coloring with  colors~1 and~2.   
Clearly, all intersecting edges can be divided into pairs of edges that intersect each other and do not intersect other edges.

Consider a pair of intersecting edges~$u_1u_2$ and~$v_1v_2$, let the vertices~$u_1$ and~$v_1$ have color~1 and the vertices~$u_2$ and~$v_2$ have color~2, let~$A$ be the point of intersection of these edges.
Without loss of generality we may assume that the edges are drawn such that their parts~$Au_1, Av_1, Au_2, Av_2$ are arranged clockwise  (see figure~\ref{ap0}b). 
Let us color these parts with the colors of their ends: $Au_1$ and~$Av_1$  with color~1, $Au_2$ and $Av_2$ with color~2.

Consider the part~$Au_1$. The part~$Av_1$ of the same color is on the right side of~$Au_1$, and the part~$Av_2$ of the other color is on the left side. Similarly for~$Au_2$: a part of the same color is on the right side  and a part of the other color is on the left side.
For the parts~$Av_1$ and~$Av_2$ one can see the converse: a part of the same color is on the left side  and a part of the other color is on the right side. This allows us to distinguish~$u_1u_2$ from~$v_1v_2$. 

\begin{defin}
We call the edge~$u_1u_2$~{\it right} and the edge~$v_1v_2$ {\it left}. We  do the same for any pair of intersecting edges.

An edge of a drawing~$\cal G$ is called {\it simple}, if it does not intersect other edges.  Let us denote by~$p({\cal G})$ the number of simple edges of a drawing~$\cal G$ and by~$t({\cal G})$ the number of pairs of intersecting edges of a drawing~$\cal G$.
\end{defin}

Thus the set of edges of a drawing~$\cal G$ is divided into three subsets: simple, right and left edges. The numbers of left and right edges are equal to $t({\cal G})={e(G)-p({\cal G})\over2}$.

Let us return to the pair of intersecting edges~$u_1u_2$ and~$v_1v_2$, described above, and formulate one more condition.\\

$5^\circ$ {\it Let~$u_1u_2$ and~$v_1v_2$ be a pair of intersecting edges, such that  vertices~$u_1$ and~$v_1$ have color~$1$ and vertices~$u_2$ and~$v_2$ have color~$2$. Then the edges~$u_1v_2$ and~$u_2v_1$ belong to~$E(G)$. Moreover, these edges are simple in the drawing~$\cal G$.}\\

Let us prove, that condition~$5^\circ$ holds for a drawing~$\cal G$ with minimum number of intersections of a maximal  graph~$G$. One can draw the edge~$u_1v_2$ almost along the path~$u_1Av_2$, and the edge~$u_2v_1$ almost along the path~$u_2Av_1$. (see figure~\ref{ap0}b). Since the graph~$G$ is maximal,~$u_1v_2,u_2v_1 \in E(G)$. If one of these edges is not simple in~$\cal G$, 
then we shall draw it as it was written above and decrease the number of intersections. We obtain a contradiction. Thus, the condition~$5^\circ$ holds.

\begin{defin}
A drawing of  a maximal almost planar bipartite graph is called {\it regular},  if it satisfies conditions~$1^\circ-5^\circ$.
\end{defin}

{\sf In what follows~$G$ is a maximal almost planar bipartite graph and~$\cal G$ is its regular drawing.}

\section{\bf Proof of the bound}

We will prove the bound from theorem~\ref{dap} by induction on the number of vertices.  The statement of the theorem for a graph on 4 vertices is trivial. It is the base of induction. {\sf In what follows we assume that the graph~$G$ contains more than~4 vertices, and the bound is  proved for any  graph with less number of vertices.}

\subsection{Left, right and simple edges}
We need several lemmas about properties of drawings.

\begin{lem}
\label{llr1}
Let~$ux_1,vx_2$ and~$uy_1,vy_2$ be two pairs of intersecting edges in a drawing~$\cal G$, such that~$vx_2$ and~$vy_2$ are right edges in their pairs and vertices~$u$ and~$v$ have different colors. Then~$e(G)\le 3v(G)-11$. 
\end{lem}

\begin{proof}  Let~$A$ be the point of intersection of the edges~$ux_1$ and~$vx_2$, and~$B$ be the point of intersection of the edges~$uy_1$ and~$vy_2$ (see figure~\ref{ap11}a).
Let us orient parts of our edges as follows:  $A\to v \to B \to u \to A$ and obtain an oriented cycle~$C$, that divide the plane into two parts. Our aim is to prove that one of these parts contains~$x_1$ and~$x_2$, and the other  contains~$y_1$ and~$y_2$.
Clearly, the vertices~$u,x_2,y_2$ have the same color, and the vertices~$v,x_1,y_1$ have the other color.   Let us go around the cycle~$C$ (in the direction mentioned above). The vertices~$x_2$ and~$x_1$ will be on the right side. (Since~$vx_2$ is a right edge and the vertices~$v$ and~$x_1$ have the same color, the path~$Ax_1$ is on the right side of the path~$Av$.) Similarly, since~$uy_1$ is a left edge, the vertices~$y_1$ and~$y_2$ will be on the left side of the cycle~$C$  (see figure~\ref{ap11}a).

\begin{figure}[!ht]
	\centering
		\includegraphics[width=1\columnwidth, keepaspectratio]{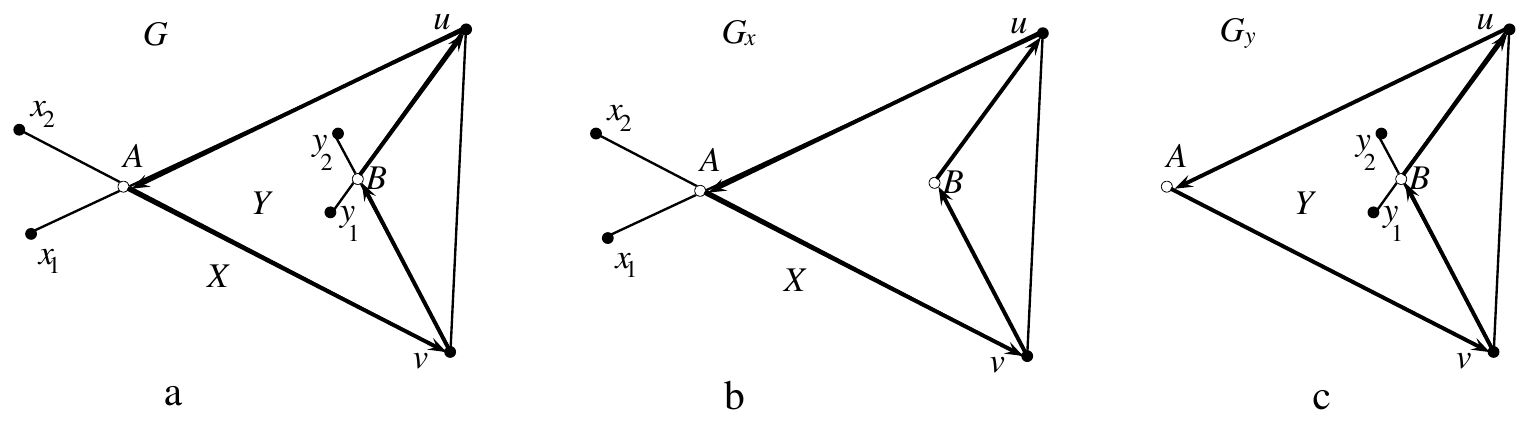}
     \caption{The graphs~$G$, $G_x$ and~$G_y$.}
	\label{ap11}
\end{figure} 

Let us cut the plane by the cycle~$C$ into the parts~$X$ (that contains~$x_1,x_2$) and~$Y$ (that contains~$y_1,y_2$). We set, that the vertices~$u$ and~$v$ belong to both parts. Let~$G_x$ be the induced subgraph of~$G$ on the vertex set~$X$ (see figure~\ref{ap11}b), and~$G_y$ be the induced subgraph of~$G$ on the vertex set~$Y$ (see figure~\ref{ap11}c). Note, that no edge of the graph~$G$ can intersect the cycle~$C$, hence~$E(G)=E(G_x)\cup E(G_y)$. Moreover, 
$$V(G_x)\cap V(G_y) =\{u,v\}, \quad  E(G_x)\cap E(G_y)=\{uv\}$$ (by condition~$5^\circ$ the vertices~$u$ and~$v$ are adjacent). 
Clearly, $$4\le v(G_x)< v(G) \quad \mbox{and}  \quad 4\le v(G_y)< v(G).$$ 
 The bound from theorem~\ref{dap} is proved for less graphs~$G_x$ and~$G_y$, hence
$$
e(G)=  e(G_x)+e(G_y)-1   \le (3 v(G_x)-8) +(3v(G_y)-8) -1 =$$
    $$ 3(v(G_x)+v(G_y)-2)-11 = 3v(G)-11.
$$
\end{proof}

\begin{rem}
If the vertices~$u$ and~$v$ in condition of lemma~\ref{llr1} would be of the same color, one can similarly prove that~$e(G)\le 3v(G)-10$. 
The only difference is that in this case the graphs~$G_x$ and~$G_y$ do not contain the edge~$uv$. However, we do not need this statement.
\end{rem}

\begin{lem}
 \label{llr}
Let~$e(G)\ge 3v(G)-10$ and~$\cal G$ be its regular drawing. Then for any vertex~$w\in V(G)$ both  numbers of right and left edges, incident to~$w$,  do not exceed the number of simple edges, incident to~$w$.
\end{lem}

\begin{proof}  Clearly, it is enough to prove the statement for right edges.
Let $wx_1,\dots,wx_k$ be all right edges of~$\cal G$, incident to~$w$.
Let the edge~$wx_i$ intersect in the drawing~$\cal G$ a left edge~$y_iz_i$, such that the vertices~$y_i$ and~$x_i$ are of the same color.
By condition~$4^\circ$ then~$wy_i$ is a simple edge of~$\cal G$.
If for some~$i$ and~$j$ the vertices~$y_i$ and~$y_j$ coincide, then by lemma~\ref{llr1} we have~$e(G)\le 3v(G)-11$. That contradicts the condition of lemma. Hence~$wy_1,\dots,wy_k$ are different simple edges and their number is at least the number of right edges, incident to~$w$.
\end{proof}

Consider a plane graph~$G'$, obtained from the drawing~$\cal G$ after deleting all left edges. Clearly,~$G'$ is a bipartite graph,
$$v(G')=v(G) \quad \mbox{and} \quad {e(G')=e( G)-t({\cal G})} .$$ 
It follows from the classic bound for bipartite planar graphs, that
\begin{equation}
 \label{e1}
 e(G')\le 2v(G')-4 = 2v(G)-4. 
\end{equation}

\begin{lem}
\label{lg'} 
Let~$\delta(G)\ge 4$, $e(G)\ge 3v(G)-10$. Then the following statements hold.

$1)$ For any vertex~$w\in V(G')$ there are at least two simple edges, incident to~$w$. Hence, $\delta(G')\ge 2$. 

$2)$ The graph~$G'$ is connected. The boundary of any face of the plane graph~$G'$ contains a simple cycle with at least~$4$ edges. 

\end{lem}

\begin{proof}  1) By lemma~\ref{llr} there are at least~$\lceil{d_G(w)\over 3}\rceil \ge 2$ simple edges, incident to~$w$.  

2) It is easy to see on the figure~\ref{ap0}b, that any left edge of a regular drawing can be replaced by a path of two simple edges and one right edge. Hence, the graph~$G'$ is connected.  Since~$\delta(G')\ge 2$, there is a cycle in~$G'$. Hence, the boundary of any face of~$G'$ contains a simple cycle. Since~$G'$ is a bipartite graph, any cycle of~$G'$ contains at least four edges.
\end{proof}

\subsection{Proof of the bound~$3v(G)-8$}

\begin{lem}
\label{3v8}
Let~$G$ be a bipartite graph with~$v(G)\ge 4$.  Then~$e(G)\le 3v(G)-8$. 
\end{lem}

\begin{proof} 
\q1.  {\it If there is a vertex of degree at most~$3$ in the graph~$G$}, then we delete from~$G$ this vertex and all edges incident to it. We obtain a bipartite graph~$G_1$ with~$v(G_1)=v(G)-1$,  which is clearly almost planar. By induction assumption we have
$$e(G_1)\le 3v(G_1)-8 = 3v(G)-8-3,$$  whence it follows the statement we want to prove for the graph~$G$.

\smallskip 
\q2.  {\it  Now we consider the case~$\delta(G)\ge 4$.}

\noindent It is enough to prove the statement for a maximal almost planar bipartite graph~$G$ with~$e(G)\ge 3v(G)-10$. Consider its regular drawing~$\cal G$ and the plane graph~$G'$ constructed above. It follows from lemma~\ref{lg'}, that for any vertex~$v\in V(G)$ there are at least two simple edges incident to~$v$. Whence it follows, that~$p({\cal G})\ge v(G)$. 

By the inequality~(\ref{e1}) we have
\begin{equation}
 \label{ep2}
2v(G)-4\ge e(G')=t({\cal G})+p({\cal G}),
\end{equation}
whence it follows, that
\begin{equation}
 \label{ep3}
t({\cal G})\le v(G)-4 \quad \mbox{ and } \quad e(G)=t({\cal G})+e(G')\le 3v(G)-8,
\end{equation}
what was to be proved.
\end{proof}

\subsection{Structure of graphs with~$e(G)=3v(G)-8$}
Let~$G$  be a bipartite almost planar graph, such that~$e(G)= 3v(G)-8$. 

\subsubsection{Case of minimal degree at least~$4$}
\label{delta4}

Let~$v(G)\ge 8$ and~$\delta(G)\ge 4$.
Consider a regular drawing~$\cal G$ and the plane graph~$G'$, that consists of simple and right edges of the drawing~$\cal G$. Clearly, the statements of lemmas~\ref{llr} and~\ref{lg'} hold for the graphs~$G$ and~$G'$.

\begin{cor}
\label{c3v8}
Let~$G$ be a bipartite almost planar graph with~${v(G)\ge 4}$, \linebreak ${\delta(G)\ge 4}$ and~$e(G)= 3v(G)-8$. Then the following statements hold.

$1)$ The graph~$G$ is maximal. Any vertex of~$G$ is incident to precisely two simple edges of  its regular drawing~$\cal G$.

$2)$ $\Delta(G')\le 4$.
\end{cor}

\begin{proof} 1) It follows from lemma~\ref{3v8}, that the graph~$G$ is maximal. Let us return to the proof of lemma~\ref{3v8}. Since~$\delta(G)\ge 4$, we have case~2 of the proof, hence,   any vertex is incident to at least two simple edges. If some vertex is incident to more than two simple edges, then~$p({\cal G})>v(G)$ and it follows from the inequalities~$(\ref{ep2})$ and~$(\ref{ep3})$, that $e(G)< 3v(G)-8$. We obtain a contradiction. 

2) Let~$v\in V(G)$. By lemma~\ref{llr} in the graph $G$ the number of simple edges, incident  to the vertex~$v$ is at least the number of right edges, incident to~$v$. Hence if~$d_{G'}(v)\ge 5$, then~$v$ is incident to at least three simple edges, that contradicts to item~1.
\end{proof}

\smallskip
\q1.   {\it Faces of the graph~$G'$ and left edges.}

\noindent
It follows from~$e(G)=3v(G)-8$, that equality must be attained in the inequality~(\ref{e1}). Hence, the boundary of any face~$D$ of the graph~$G'$ contains precisely~4 edges. By lemma~\ref{lg'}, this boundary contains a cycle with at least~4 edges. Thus, the boundary of~$D$ is a simple cycle of length~4.  Whence it follows that the graph~$G'$ is biconnected.

{\sf Thus we deal with biconnected plane graph~$G'$. Any face of~$G'$ is a quadrangle and~$\Delta(G')\le 4$.}

Consider left edges of the graph~$G$, that we have deleted. Any left edge~$f$ intersect exactly one right edge, which divides~$f$ into two {\it half-edges}.
These half-edges lie in two neighboring faces of~$G'$, and the common part of their bounds contains the right edge which intersects~$f$.

\begin{figure}[!ht]
	\centering
		\includegraphics[width=0.6\columnwidth, keepaspectratio]{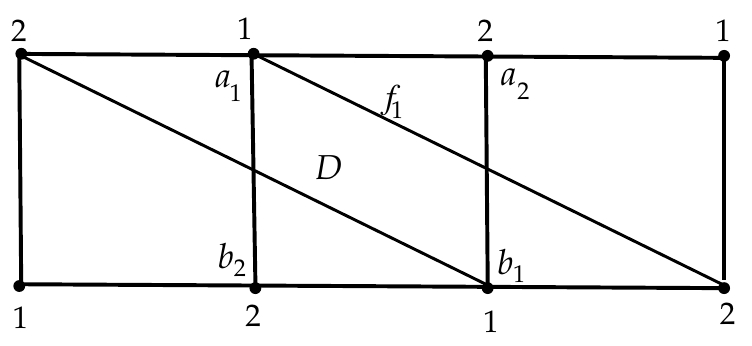}
     \caption{Left edges on faces-quadrangles of the graph~$G'$.}
	\label{ap3}
\end{figure} 

\smallskip

Let~$D=a_1a_2b_1b_2$ be a face of the graph~$G'$, the vertices~$a_1$ and~$b_1$ have color~1 and the vertices~$a_2$ and~$b_2$ have color~2. 
Assume, that this face contains a half-edge~$f_1$ (a part of a left edge~$f$) with the end~$a_1$. 

The other end of the half-edge~$f_1$ is on an edge~$e$ of the face~$D$. Since edges~$e$ and~$f$ do not have a common end and~$f$ is a left edge,  it is easy to see, that~$e=a_2b_1$ (see figure~\ref{ap3}).

Thus we have proved, that   {\sf not more than one half-edge lying in~$D$ can be incident to a vertex of the face~$D$. Moreover, the other end of this half-edge must lie on the certain edge.}
A half-edge incident to~$a_1$ must have the other end on the edge~$a_2b_1$, a half-edge incident to~$a_2$ must have the other end on the edge~$b_1b_2$, a half-edge incident to~$b_1$ must have the other end on the edge~$b_2a_1$, and a half-edge incident to~$b_2$ must have the other end on the edge~$a_1a_2$.

Let us return to the half-edge~$f_1$, lying in~$D$,  that is incident to~$a_1$ and have an end on the edge~$a_2b_1$. Since two different half-edges, lying in~$D$, cannot intersect each other, at least one other half-edge can lie in~$D$~--- a half-edge incident to~$b_1$ with  the other end on the edge~$a_1b_2$ (see figure~\ref{ap3}). 

\begin{rem}
\label{rpr}
We have that the boundary of any face of the graph~$G'$ contains at most two right edges.  If it contains two right edges, then they are opposite edges of this boundary. 
\end{rem}

\smallskip

\q2. {\it Strips and rings.}

\noindent
Let us build an auxiliary graph~$F$, which vertices are faces of the graph~$G'$ and two vertices are adjacent if and only if the boundaries of correspondent faces have a common right edge. By remark~\ref{rpr}, the degree of any vertex in the graph~$F$ is at most two, hence, this graph is a union of several simple cycles and simple paths. Clearly,~$e(F)$ is equal to the number of right edges of~$\cal G$, i.e., to $t({\cal G})=v(G)-4$. The number of vertices of~$F$ is equal to the number of faces of~$G'$,  which is by Euler's formula equal to~$v(G)-2$. Thus, the graph~$F$ consists of two paths and, possibly, several cycles.

Let us cut the plane  by drawings of all simple edges of~$G'$ (along these edges). {\it As a result of cutting, the plane is divided into several parts and each simple edge is divided into two edges with the same ends} --- we call these edges by {\it halves}.

{\sf We consider two halves of one simple edge as different edges even in the case when these halves belong to the same part of plane!}

A pair of faces of~$G'$ has a common part of boundary after cutting if and only if these faces had a common right edge before cutting. Thus, {\sf any part of the plane formed after cutting consists of all faces of some connected component of the graph~$F$.} 

\begin{defin}
Parts of plane, correspondent to components-paths of the graph~$F$, are called {\it strips}.  Parts of plane, correspondent to components-cycles of the graph~$F$, are called  {\it rings}.

A {\it boundary edge} of a strip or a ring is any half of a simple edge, which belongs to this strip or ring. 

A {\it boundary}  of a strip or a ring is any connected component of the subgraph consisting of boundary edges of  this strip or ring. 

Any strip is a plane graph, isomorphic to a rectangle with side~1, divided into cells $1\times 1$. We call four vertices of the strip, correspondent to angles of this rectangle, by {\it corners} of the strip. We call faces of the strip, that contain its corners, by {\it extreme faces}.

All strips and rings, obtained from the graph~$G'$ after cutting, are called  {\it parts of the graph~$G'$}.
\end{defin}

\begin{rem}
\label{kopo}
1) It follows from the proved above, that there are precisely two strips among the parts of the graph~$G'$. 

2) Operating in the reverse order, one can glue the graph~$G'$ from its parts (strips and rings).

3) A strip with~$n$ faces has one boundary, which is a cycle of length~$2n+2$. A ring with~$n$ faces has two boundaries, and each of these boundaries is a cycle of length~$n$.
\end{rem}

Let~$P$ be a graph on the vertex set~$V(G)$, which edges are simple edges of~$G'$.  By corollary~\ref{c3v8} any connected component of the graph~$P$ is a simple cycle. 

It is easy to see, that edges, which halves form some boundary~$Z$, are edges of one connected component of the graph~$P$, i.e. edges of a simple cycle. We denote this cycle by~$C(Z)$. Since~$C(Z)$ is a simple cycle, {\sf the boundary~$Z$ cannot contain two halves of one edge.}

Let~$Z_1,\dots, Z_k$ be all boundaries of the parts of~$G'$.  It follows from written above, that boundaries~$Z_1,\dots, Z_k$ can be divided into pairs~$Z_i, Z_j$, for which~$C(Z_i)=C(Z_j)$. We call such boundaries~{\it congruent}.

Let us construct an auxiliary graph~$L$, which vertices are parts of the graph~$G'$ (i.e., strips and rings), and two vertices are adjacent if and only if the correspondent parts have congruent boundaries. Clearly, the graph~$L$ is connected. It follows from proved above, that two parts with common boundary edge  have congruent boundaries. Hence the degree of any ring in the graph~$L$ is at most~2, and the degree of any strip is at most~1. Consequently, the graph~$L$ is a chain, which leaves are two strips of the graph~$G'$, and all other vertices (if such vertices exist) are rings. 

Clearly, two boundaries of a ring have equal numbers of vertices, and two congruent boundaries too. Hence, all boundaries have equal numbers of vertices. Let one of the strips have~$k$ faces.  Then the length of its boundary is equal to~${2k+2}$, as soon as the length of all other boundaries. Since any vertex of~$G'$ belongs to exactly two boundaries, we have, that the number of vertices of the graph~$G'$ is equal to~$(2k+2)\cdot e(L)$, i.e., this number is even.

\smallskip
{\sf Thus we have proved, that in the case where~$e(G)=3v(G)-8$ and~${\delta(G)\ge 4}$ the number of vertices of the graph~$G$ is even.}

\subsubsection{Case of minimal degree at most~$3$}
\label{delta3}

  Let~$H$ be a bipartite almost planar graph with
$$v(H)\ge 5, \quad e(H)=3v(H)-8, \quad \delta(H)\le  3.$$  

\noindent  By lemma~\ref{3v8} the graph~$H$ is maximal, let us consider its regular drawing~$\cal H$.
Starting with the graph~$H$, we  delete from the graph we have a vertex of   degree at most three until we obtain a bipartite almost planar graph~$G$ either with~$v(G)=4$, or with~$\delta(G)\ge 4$. In both cases we have~$e(G)\le 3 v(G)-8$. Since~$e(H)=3v(H)-8$,  each time we have deleted a vertex of degree~3 and~$e(G)= 3 v(G)-8$. 
It remains to prove that we cannot add a vertex of degree~3 to the graph~$G$, such that the resulting graph remains bipartite and almost planar. 

In the case~$v(G)=4$ the graph~$G$ is a cycle with~4 vertices, for which the statement we prove is obvious.
Let~$v(G)>4$. Clearly, the graph~$G$ is maximal. Let us delete from the regular drawing~$\cal H$ all vertices of~${V(G)\setminus V(H)}$. We obtain a regular drawing~$\cal G$ of the graph~$G$ (it is easy to see, that conditions~$1^\circ-5^\circ$ do not break after deleting vertices).

We have proved, that~$v(G)$ is even and studied the structure of such graphs. Let us consider the subgraph~$G'$  with simple and right edges of~$G$. All faces of~$G'$ are quadrangles. Let the vertex~$a$ which we want to add lie on a face~$D$. 
We need to join the vertex~$a$ to three vertices of~$G$ of the same color. Consider three cases.

\smallskip \goodbreak
\q{a}. {\it $D$ is not an extreme face of a strip.}

\noindent Then without loss of generality we may assume, that two left edges are incident to  vertices of color~2 of this face and intersect two opposite edges of the boundary of~$D$.  These left edges divide~$D$ into three parts. If~$a$ is disposed in one of two parts-triangles (see figure~\ref{ap2}a, let it be the right triangle), then the vertex~$a$ can be adjacent only to the vertices of the face~$D'$, neighboring to~$D$. (In this case  edges incident to~$a$ can intersect only the common edge of~$D$ and~$D'$, otherwise the graph would be not almost planar.) But there are no three vertices of the same color in~$D'$. 

If~$a$ lies in the central part of the face~$D$ (see figure~\ref{ap2}a), then~$a$ can be joined to only two vertices: the vertices of color~2 of the face~$D$.

\begin{figure}[!ht]
	\centering
		\includegraphics[width=\columnwidth, keepaspectratio]{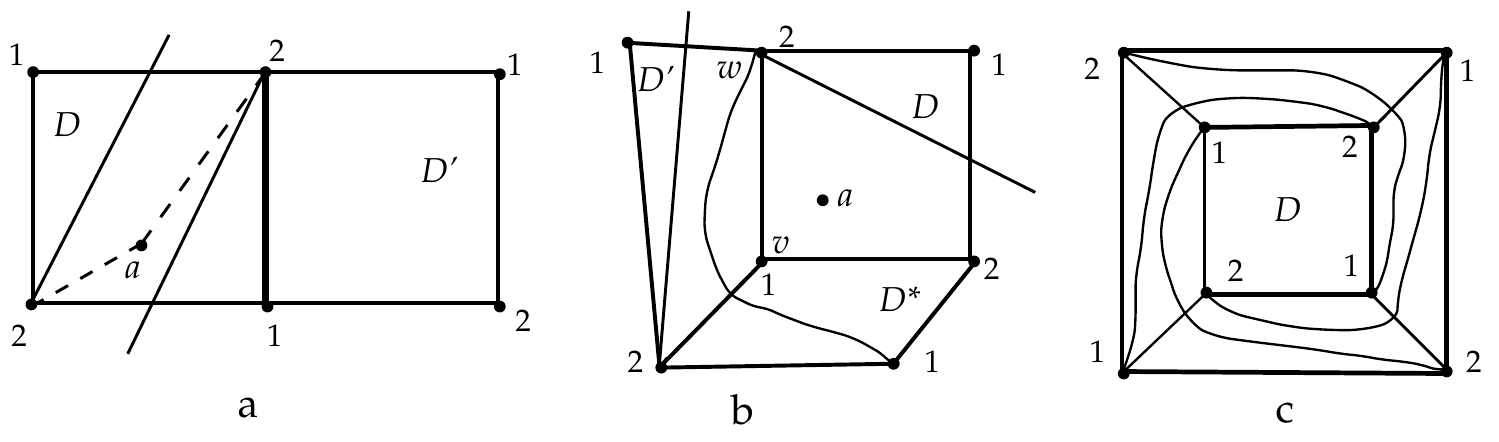}
     \caption{Disposition of the vertex~$a$.}
	\label{ap2}
\end{figure} 

\smallskip 
\q{b}. {\it $D$ is an extreme face of a strip, consisting of more than one face.}

\noindent 
We need the following simple lemma.

\begin{lem}
\label{ugly} A vertex~$x\in V(G')$ is a corner of a strip precisely~$4-d_{G'}(x)$ times.
\end{lem}

\begin{proof}
 Consider a face~$M$, which boundary contains the vertex~$x$. Clearly, both edges of~$M$ incident to~$x$ are simple if and only if~$M$ is an extreme face of a strip and~$x$ is a corner of this strip. In any other case one of these two edges is simple and the other one is right. 

By corollary~\ref{c3v8}, the vertex~$x$ is incident to exactly two simple edges. That immediately implies  the statement we prove for the cases, where~$d_{G'}(x)=2$ or~$d_{G'}(x)=3$. If~$d_{G'}(x)=4$, then~$x$ is incident to two simple and two right edges, and, clearly, these edges alternate in plane  drawing. Hence a vertex of degree~4 cannot be a corner of a strip. 
\end{proof}

Let the face~$D$  have corners~$w$ and~$v$.
Clearly, precisely one left edge of the graph~$G$ intersects the face~$D$ (see figure~\ref{ap2}b). This left edge is incident to one of two corners of~$D$  (let it be the corner~$w$ of color~2). Let us show, that in this case the second corner~$v$ has degree~$d_{G'}(v)=3$. 

It follows from lemma~\ref{ugly}, that~$2\le d_{G'}(v)\le 3$. 
Let~$d_{G'}(v)=2$. Then in the graph~$G$ the vertex~$v$ is incident to two simple edges and two left edges, as it is proved in lemma~\ref{llr}. Hence, each of two faces of~$G'$, which boundary contains~$v$, must be  intersected by a half-edge, incident to~$v$. But there is no such half-edge in the face~$D$. We obtain a contradiction.

Thus,~$d_{G'}(v)=3$. Consider two different from~$D$ faces, which contain~$v$: let it be~$D'$ (which contains~$w$) and~$D^*$, see figure~\ref{ap2}b.  By lemma~\ref{ugly}, the common edge of $D'$ and $D^*$ is a right edge and is intersected by a left edge. There is the only possibility to draw this left edge: it must join the vertex~$w$ of the face~$D'$ to a vertex of color~1 of the face~$D^*$ (clearly, this is the vertex of~$D^*$, opposite to~$v$, see figure~\ref{ap2}b).

The case where the vertex~$a$ is disposed in a part-triangle of the face~$D$, is similar to the case considered above. Let~$a$ lies in the part of~$D$, which contains the corners~$v$ and~$w$. Since edges incident to~$a$ cannot intersect a left edge, that lie in the faces~$D'$ and~$D^*$, it is easy to see, that the vertex~$a$ cannot be joined to three vertices of the same color.

\smallskip 
\q{c}. {\it $D$ is the only face of some strip.}
\noindent

Hence the graph~$G'$ consists of two strips (each of them consists of one face, i.e., has four vertices), and, possibly, several rings. Since~$v(G)>4$, among the parts of~$G'$ there is at least one ring. Clearly, all rings of~$G'$ has 4 faces and each strip is glued together with a ring (see figure~\ref{ap2}c). Then it is easy to see, that the  vertex~$a$ disposed on the face~$D$ cannot be adjacent to more than two vertices of the same color.

\smallskip {\sf Thus we have proved that an almost planar graph~$H$ with~$e(H)=3v(H)-8$,  $v(H)>4$ and~$\delta(G) \le 3$ does not exist.}

\subsection{The end of the proof of the bound}

It follows from conclusions of two previous sections, that {\sf a bipartite almost planar graph~$G$ with~$e(G)=3v(G)-8$ has even number of vertices.} (It is proved in section~\ref{delta3}, that~$\delta(G)\ge 4$. It is proved for such graphs in section~\ref{delta4}, that~$v(G)$ is even.)

Note, that for each~$v\ge 4$ we denote by~$\beta(v)$ the maximal number of edges in a bipartite almost planar graph on~$v$ vertices. We have proved, that~$\beta(v)\le 3v-8$ for all~$v\ge 4$   and~$\beta(v)\le 3v-9$ for odd~$v$.

It remains to note, that~$\beta(6)\le 9=3\cdot 6-9$, since maximal number of edges in a bipartite graph on~6 vertices is attained at complete bipartite graph~$K_{3,3}$ and is equal to~9.

The proof of the bound is finished.

\section{Extremal examples}

Let us consider a parallelepiped~$1\times 1\times k$ on the coordinate grid. All nodes of the grid lying inside this parallelepiped and on its boundary  are vertices of the graph~$G'_k$, and unit segments between these nodes are edges of~$G'_k$. Clearly, the graph~$G'_k$ can be glued from two strips (upper and lower quadrangles~$1\times k$) and one ring, containing remaining~${2k+2}$ faces of the graph~$G'_k$. It is easy to draw left edges in the strips and in the ring. Clearly, we obtain an almost planar bipartite graph~$G_k$ with~$v(G_k)=4k+4$, $e(G_k)=3\cdot v(G_k)-8$ and without multiple edges as a result. Thus, we have shown that~$\beta(v)=3v-8$ for ${v=4\ell\ge 8}$.

\begin{figure}[!ht]
	\centering
		\includegraphics[width=0.7\columnwidth, keepaspectratio]{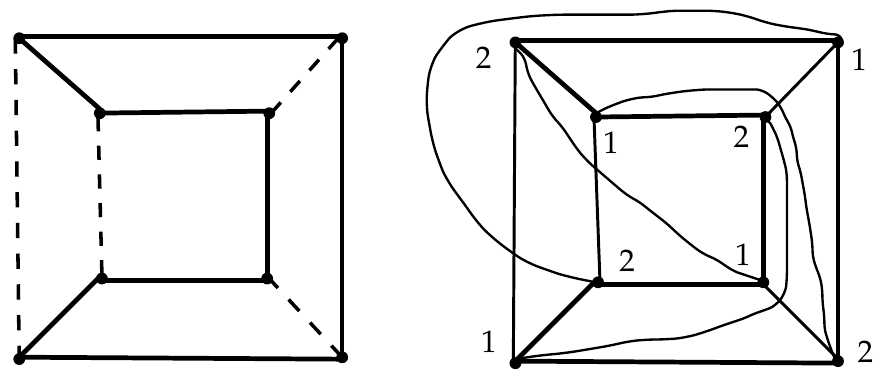}
     \caption{Cube~$1\times 1\times 1$  and almost planar graph~$K_{4,4}$.}
	\label{ap4}
\end{figure} 

Note, that a parallelepiped~$1\times 1\times k$ can be glued from two strips. An example for a cube~$1\times 1\times 1$ is shown on figure~\ref{ap4}: on the left side one can see dividing  the  cube into two strips, and on the right side~--- the resulting almost planar bipartite graph. Note, that this graph is complete bipartite graph~$K_{4,4}$.

Moreover, one can glue a parallelepiped~$1\times k \times n$ from two strips~$1\times k$  and~$n$ rings. On this base one can construct a bipartite almost planar graph~$G_{k,n}$ with~$v(G_{k,n})=2(k+1)(n+1)$ and~$e(G_{k,n})= 3v(G_{k,n})-8$. However, even numbers of type~$2p$, where~$p$ is prime, cannot be represented as~$2(k+1)(n+1)$. Thus we need to construct more complicated examples.

\begin{figure}[!ht]
	\centering
		\includegraphics[width=\columnwidth, keepaspectratio]{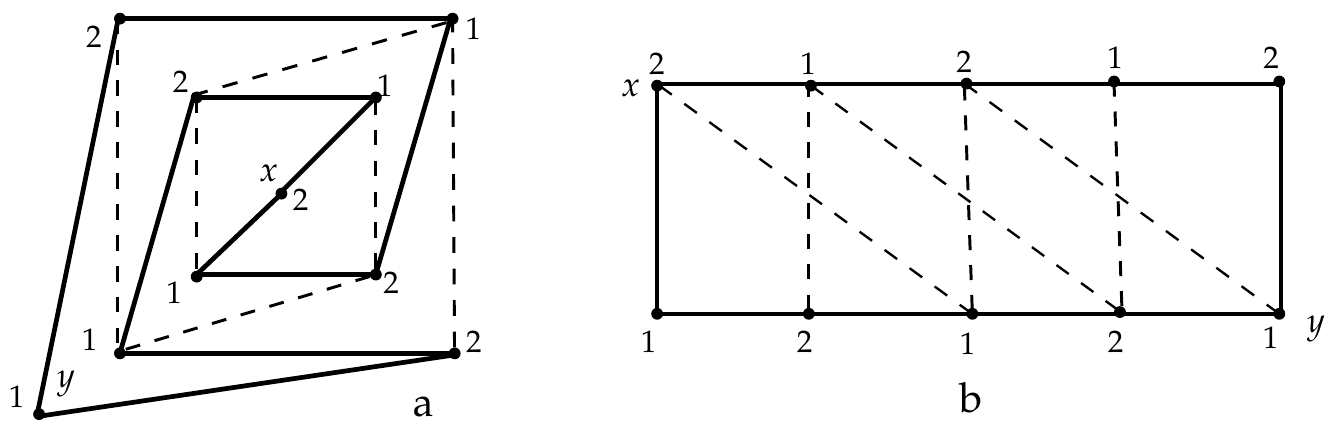}
     \caption{Gluing an almost planar graph from two strips.}
	\label{ap5}
\end{figure}

Let us construct an almost planar graph~$H_k'$ on~$4k+6$ vertices (where ${k\ge 1}$) from two strips~$1\times (2k+3)$. An example for~$k=1$ is shown on figure~\ref{ap5}a. In the graph~$H_1'$ shown on this picture there are two vertices of degree~2 (they have different colors, let them be~$x$ of color~2 and~$y$ of color~1), four vertices of degree~3 (two vertices of each color), all other vertices have degree~4. Simple edges of the graph~$H_1'$  are shown as solid lines, right edges are shown as dotted lines.  It is shown how to glue this graph from two strips~$1\times 4$. It is not difficult to add on this graph to a graph~$H_k'$ on~$4k+6$ vertices (by adding~$k-1$ new squares). 

It is shown on the  figure~\ref{ap5}b, how to draw left edges in each strip. It is important, that in both strips --- parts of~$H'_k$ ---$x$ and~$y$ are two corners, incident to left edges. It is easy to verify, that for~$k\ge 1$ the resulting graph has no multiple edges. 

We draw left edges in the graph~$H_k'$  as it is shown and obtain an almost planar bipartite graph~$H_k$ with~$v(H_k)=4k+6$ and~$e(H_k)=3v(H_k)-8$. Thus we have proved, that~$\beta(v)=3v-8$ for~$v=4k+6$, where~$k$ is a positive integer. 

\begin{rem}
1)  Note, that if we try to construct such a graph for~$k=0$ then  the edge~$xy$ would be drawn in both strips, hence, the graph would have multiple edges. 

2) If we draw in any strip right edges instead of left edges (i.e., the other diagonals of rectangles~$1\times 2$), then multiple edges appear. One of them is the edge joining the corner of this strip which has degree~3 in the graph~$H_k'$ to the correspondent vertex, another one is a side of upper or lower square.

2) Note also that there is a lot of other almost planar bipartite graphs, that can be constructed from two strips of the same length.
\end{rem}

\smallskip
Let us consider small values of~$v$. For~$v=4$ and~$v=5$ extremal examples are complete bipartite graphs~$K_{2,2}$ and~$K_{2,3}$, respectively. These graphs are planar, and, hence, they are almost planar graphs, too. Note, that~$e(K_{2,2})=4=3\cdot 4-8$ and~$e(K_{2,3})=6=3\cdot 5-9$.  Thus we have proved, that~$\beta(4)=4$ and~$\beta(5)=6$.

For~$v=6$ and~$v=7$ extremal examples are complete bipartite graphs~$K_{3,3}$ and~$K_{3,4}$, respectively. These graphs are almost planar, since the graph~$K_{4,4}$, shown on figure~\ref{ap4}, is almost planar. Note, that~$e(K_{3,3})=9=3\cdot 6-9$ and $e(K_{3,4})=12=3\cdot 7-9$.  Thus we have proved, that~$\beta(6)=9$ and~$\beta(7)=12$.

It remains to show, that~$\beta(v)=3v-9$ for odd~$v\ge 9$.  For this it is enough to add one  vertex and two edges to an almost planar bipartite graph with~$v-1$ vertices and~$3(v-1)-8$ edges (note, that~$v-1\ge 8$ is even, hence, such a graph was constructed before). One can add a vertex~$a$ to this graph and join it to two vertices of the same color  as it is shown on figure~\ref{ap2}a.

The proof of the theorem is finished.

\section{On complete bipartite graphs}

In the end of our paper we turn attention to complete bipartite graphs and find out all almost planar graphs among them.
The graphs~$K_{1,n}$ and~$K_{2,n}$ are planar, and, consequently, almost planar. We have  ascertained, that~$K_{3,3}$, $K_{3,4}$ and~$K_{4,4}$ are almost planar graphs.

It is clear from figure~\ref{ap6}, that the graph~$K_{3,6}$ is almost planar. Hence its subgraph~$K_{3,5}$ is almost planar, too. Since~$e(K_{4,5})=20>3\cdot 9-9$, the graph~$K_{4,5}$ is not almost planar.

\begin{figure}[!ht]
	\centering
		\includegraphics[width=0.4\columnwidth, keepaspectratio]{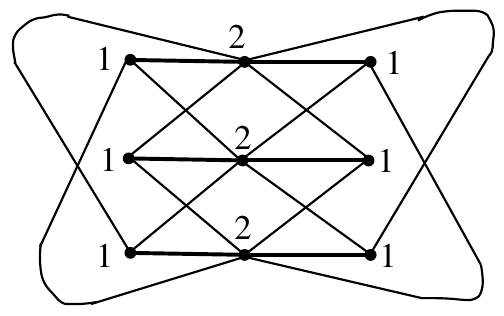}
     \caption{The graph~$K_{3,6}$.}
	\label{ap6}
\end{figure} 

Let us show, that the graph~$K_{3,7}$ is not almost planar, i.e., this graph cannot be drawn on the plane such that each edge intersect at most one other edge.

At first we note, that~$\cr(K_{3,3})\ge 1$. Let us draw the graph~$K_{3,5}$  on the plane. It contains 10 subgraphs, isomorphic to~$K_{3,3}$, and the drawing of each of these subgraphs contains at least one crossing of edges. It is easy to see, that each pair of crossing edges belongs to precisely~$C_{3}^1=3$ subgraphs~$K_{3,3}$.  Hence, $\cr(K_{3,5})\ge {10\over 3}$ and, consequently,~$\cr(K_{3,5})\ge 4$.

Now let us assume, that~$K_{3, 7}$ is an almost planar graph. Clearly, this graph is maximal (since it is a complete bipartite graph). Consider its regular drawing. Since~$e(K_{3,7})=21>3v(K_{3,7})-10$, by lemma~\ref{llr} each of three vertices of degree~7 is incident to at least~$\lceil {7\over 3} \rceil =3$ simple edges. Hence, our drawing has at least 9 simple edges and~$\cr(K_{3,7})\le {21-9\over 2}=6$.

However, there are at least~4 crossings in the drawing of each subgraph, isomorphic to~$K_{3,5}$. There are~$C_7^5=21$ such subgraphs, and each crossing belongs to~$C_5^3=10$ subgraphs of type~$K_{3,5}$. Whence it follows, that 
$$\cr(K_{3,7}) \ge  \cr(K_{3,5}) \cdot{21\over 10} = {84\over 10}> 6.$$ 
Obtained contradiction shows us, that~$K_{3,7}$ is not an almost planar graph.

Thus,  $K_{1,n}$, $K_{2,n}$, $K_{3,3}$, $K_{3,4}$, $K_{3,5}$, $K_{3,6}$ and~$K_{4,4}$ are all almost planar graphs  among  complete bipartite graphs.

\end{document}